\documentclass[10pt,a4paper]{amsart}
\usepackage[english]{babel}
\usepackage[utf8x]{inputenc}
\usepackage{ucs}
\usepackage{tikz}
\usepackage{amsmath}
\usepackage{amsthm}
\usepackage{amsfonts}
\usepackage{amssymb}
\usepackage{hyperref}
\usepackage{dsfont}
\usepackage{mathrsfs}
\usepackage{eqnarray}
\write18{}
\usepackage{xcolor}
\usepackage{graphicx}
\usepackage{pdfpages} 
\usepackage[nothm]{thmbox}
\usepackage{esint}
\usepackage[all]{xy}
\usepackage{faktor}

\hypersetup{
    colorlinks=true,
    linkcolor=red,
    filecolor=magenta,      
    urlcolor=cyan,
}

\numberwithin{equation}{section}

\newcommand{\D}{\mathcal{D}}

\newcommand{\QQ}{\mathbb{Q}}

\newcommand{\ZZ}{\mathbb{Z}}
\newcommand{\RR}{\mathbb{R}}
\newcommand{\CC}{\mathbb{C}}
\newcommand{\NN}{\mathbb{N}}

\newcommand{\tends}[1]{\underset{#1}{\longrightarrow}}

\newcommand{\quotient}[2]{{\raisebox{.2em}{$#1$}\left/\raisebox{-.2em}{$#2$}\right.}}

\makeatletter
\DeclareRobustCommand*{\mfaktor}[3][]
{
   { \mathpalette{\mfaktor@impl@}{{#1}{#2}{#3}} }
}
\newcommand*{\mfaktor@impl@}[2]{\mfaktor@impl#1#2}
\newcommand*{\mfaktor@impl}[4]{
   \settoheight{\faktor@zaehlerhoehe}{\ensuremath{#1#2{#3}}}%
   \settoheight{\faktor@nennerhoehe}{\ensuremath{#1#2{#4}}}%
      \raisebox{-0.5\faktor@zaehlerhoehe}{\ensuremath{#1#2{#3}}}%
      \mkern-4mu\diagdown\mkern-5mu%
      \raisebox{0.5\faktor@nennerhoehe}{\ensuremath{#1#2{#4}}}%
}
\makeatother

\DeclareMathOperator{\SL}{SL}

\DeclareMathOperator{\Diff}{Diff}

\theoremstyle{definition}

\newtheorem{definition}[equation]{Definition}
\newtheorem{remark}[equation]{Remark}

\theoremstyle{theorem}

\newtheorem{theoreme}[equation]{Theorem}
\newtheorem{proposition}[equation]{Proposition}

\newtheorem{conjecture}{Conjecture}
\newtheorem{lemma}[equation]{Lemma}

\newtheorem{problem}{Problem}

\title{$\mathrm{SL}_2(\mathbb{R})$-dynamics on the moduli space of one-holed dilation tori}

\author{Adrien Boulanger}
\author{Selim Ghazouani}

\thanks{The first author was partially founded by the ERC n°647133 'IChaos'.}

\date{}

\begin{document}

\maketitle

\begin{abstract}
We study the $\mathrm{SL}_2(\mathbb{R})$-action on the moduli space of (triangulable) dilation tori with one boundary component. We prove that every orbit is either closed or dense, and that every orbit of the Teichmüller flow escapes to infinity.
\end{abstract}

\section{Introduction}

Dilation surfaces form an interesting class of geometric structures on surfaces. Part of their richness comes from the diversity of viewpoints from which they can be seen and the variety of a priori distinct topics they relate to. 

\vspace{2mm}


\textbf{$(G,X)$-structures.} They are defined as a class of $(G,X)$-structures in the sense of Thurston where $G$ is the affine group and $X$ is the complex plane. It can be convenient to think of them as a subset of \textit{(branched) complex projective structure} corresponding to restricting the structural group (see \cite{Dumas} for an introduction to complex projective structures). From this point of view, it is natural to wonder about the geometric properties of individual dilation surfaces and to try and find algebraic and geometric invariants to tell them apart. \\


\textbf{Teichmüller theory.} A dilation surface can also be seen as a holomorphic differential on a Riemann surface taking values in a flat line bundle and as such they are directly connected to \textit{Teichmüller theory}. In particular, this connection is responsible for the well-definedness of \textit{moduli spaces} of such objects. In this line of thought, one can ask about the geometry of these moduli spaces, the properties of its natural complex, algebraic and foliated structures as well as potential compactifications. Evidence suggests that the geometry of these moduli spaces could resemble that of infinite volume hyperbolic manifolds. \\


\textbf{Homogeneous dynamics}. As for translation surfaces, these moduli spaces come with a locally free $\mathrm{SL}_2(\mathbb{R})$-action which makes for a connection with homogeneous dynamics. A systematic study of this action should be compared with the recent body of work on the homogeneous dynamics of infinite volume hyperbolic $3$-manifolds (see \cite{McMullenMohammadiOh}, \cite{McMullenMohammadiOh2} and \cite{BenoistOh}) as well as to the case of the $\mathrm{SL}_2(\mathbb{R})$-action on translation surfaces (see \cite{EskinMirzakhaniMohammadi} and \cite{EskinMirzakhani}). \\


\textbf{Renormalisation theory.} Part of the interest in these $\mathrm{SL}_2(\mathbb{R})$-actions comes from the fact that they provide a renormalisation scheme for certain flows on surfaces. Indeed, dilation surfaces are dynamical objects: they carry natural families of directional foliations which are geometric realisations of  \textit{affine interval exchange transformations} (see \cite{CamelierGutierrez}, \cite{BressaudHubertMaass} and \cite{MarmiMoussaYoccoz}).  Contrary to translation flows, dilation flows display enough of the variety of behaviour that can observed for smooth flows on surfaces to be a credible finite dimensional approximation of parameter spaces of such flows (see \cite{DFG} \cite{BFG} or \cite{BowmanSanderson}). Furthermore there is some evidence that these affine interval exchange transformations play a central role in the \textit{renormalisation theory} of non-linear flows on surfaces/generalised interval exchange maps (\cite{CunhaSmania}). Dilation surfaces thus provides a geometric realisation of both parameter spaces of flows on surfaces and their renormalisation operators.


The present article is an attempt at understanding a class of dilation surfaces which is a low-dimensional toy-model for the general case. The case under scrutiny is that of {\bf dilation tori with one boundary component}. We cobble together a study of their geometric properties, description of some geometric features of the moduli space, a complete analysis of the dynamics of their directional foliations and a study of the $\mathrm{SL}_2(\mathbb{R})$-action as well as that of the Teichmüller flow. \\

\noindent We denote by $\mathcal{D}$ the moduli space of dilation tori (see Section \ref{subsecdefmoduli} for a precise definition) with one geodesic boundary component. Our main Theorem is the following

\begin{theoreme}[Main theorem]

Let $T$ be an element of $\mathcal{D}$. 

\begin{itemize}

\item If the linear holonomy of $T$ generates a discrete subgroup of $(\mathbb{R}_+, \times)$, then $\mathrm{SL}_2(\mathbb{R}) \cdot T$ is closed.

\item If the linear holonomy of $T$ does not generates a discrete subgroup of $(\mathbb{R}_+, \times)$, then $\mathrm{SL}_2(\mathbb{R}) \cdot T$ is dense in $\mathcal{D}$.

\item The orbit of $T$ under the action of the Teichmüller flow diverges (eventually leaves every compact of $\mathcal{D}$).

\end{itemize}

\end{theoreme}

\noindent We also prove that all trajectories of the Teichmüller flow escape to infinity. Together with the above theorem, this supports the analogy between moduli spaces of dilation surfaces and infinite volume hyperbolic manifolds. We discuss this further in Section \ref{comments}

\vspace{3mm}

\paragraph*{\bf Structure of the article} The article in organised as follows: in Section \ref{secdilsurface}, we introduce basic material about dilation surfaces. In Section \ref{secroom}, we specify some of the material of Section \ref{secdilsurface} to the case of one-holed dilation tori and introduce a set of coordinates on the moduli space upon which we build to study the action of $\mathrm{SL}_2(\mathbb{R})$ on $\mathcal{D}$. In Section \ref{directionalfoliations}, we reproduce some material from \cite{BFG} and apply it to give a systematic description of the dynamics of directional foliations on one-holed dilation tori. We turn to describing orbits of the Teichmüller flow in Section \ref{secTeich}. Finally, we conclude the article with comments, open questions and conjectures in Section \ref{comments}. \\

\section{dilation surfaces and their moduli spaces, generalities.}
\label{secdilsurface}

\subsection{Generalities on dilation surfaces}  The main objects we will deal with in this article are dilation structures, defined as follows.

\begin{definition}
	\label{def affine surface}
  			A \textbf{dilation structure} on a topological surface $\Sigma$ - possibly with boundary -  is given by a finite set $S \subset \Sigma$, the \textbf{singularities} of  $\Sigma$, and an atlas of charts
  $\mathcal{A} = (U_i, \varphi_i)_{i \in I}$ on $\Sigma \setminus S$
such that
	\begin{itemize}
  		\item the transition maps are locally restriction of elements of
    	$\mathrm{Aff}_{\RR^*_+}(\CC) = \{ z \mapsto az+b \ | \ a \in \RR^
    *_+ \ , \ b \in \CC \} $;
       \item Seen in any chart, any connected component of the boundary - if any - must be a straight line;
  		\item each singularity in the interior of $\Sigma$ has a punctured neighborhood which is
    affinely equivalent to a punctured neighborhood of the cone point
    of a Euclidean cone of angle multiple of $2\pi$;
    	\item each singularity on the boundary of $\Sigma$ has a punctured neighborhood which is
    affinely equivalent to a punctured neighborhood of the cone point
    of an Euclidean cone.
  
	\end{itemize}
\end{definition}

Note that the notion of straight line in this setting is well-defined since changes of coordinates are affine maps, as opposed to the notion of geodesic which requires an invariant metric. Moreover, any direction $\theta \in \mathbb{S}^1$ the foliation by straight lines of $\CC$ in the direction defined by $\theta$ being invariant by dilation maps, it gives rise to a well-defined oriented foliation $\mathcal{F}_{\theta}$ on any dilation surface. Such a foliation is called a \textbf{directional foliation}. We call the \textbf{directional foliations} the resulting family of foliations, therefore indexed by the circle, that we denote $(\mathcal{F}_{\theta})_{\theta \in \mathbb{S}^1}$. \\

Finally, note  that there is a canonical complex structure induced by a dilation structure, since dilations are in particular holomorphic maps.  \\

In order to define the moduli space of dilation structure, we need to consider them up to trivial transformations, which consists to push-forward a dilation structure via a diffeomorphism, artificially giving rise to two different atlases on a given topological surface.

\begin{definition}
	\label{def affine auto}
Given two dilation structures $\mathfrak{D}_1$, $\mathfrak{D}_2$ on a topological surface $\Sigma$, we will say that a diffeormorphism $g \in \Diff^+(\Sigma)$ which preserves the boundary of $\Sigma$ is an \textbf{dilation automorphism} with respect to the pair ($\mathfrak{D}_1$, $\mathfrak{D}_2$) if, written in any affine coordinates of $\mathfrak{D}_1$ and $\mathfrak{D}_2$, $g$ writes like a dilation map.
\end{definition}

\subsection{Triangulation and polygonal models}
\label{subsec Triangulation and polygonal models}
On the contrary of translation surfaces, not all dilation surfaces carry a geometric polygonation. The obstruction of carrying such a triangulation can be characterize with the notion of affine cylinder. The following definition will also be of crucial importance in order to state the two divergence criteria given in Section \ref{secTeich}.

\begin{definition} 
	\textbf{A dilation cylinder of angle $\theta$ and multiplier $\lambda$}, denoted by $C_{\theta}^{\lambda}$, is the quotient of an angular sector of angle $\theta$ in $\CC^*$  by the action of a dilation $ z \mapsto \lambda z$ with $\lambda \in \RR^*_+ \setminus \{1\}$. 
\end{definition} 

Note that the construction only makes sense for $\theta \leq 2\pi$. However, replacing $\CC^*$  with an infinitely branched cover we can extend the construction to any positive $\theta$. \\

\begin{figure}[!h]
	\begin{center}
		\def\svgwidth{0.6 \columnwidth}
			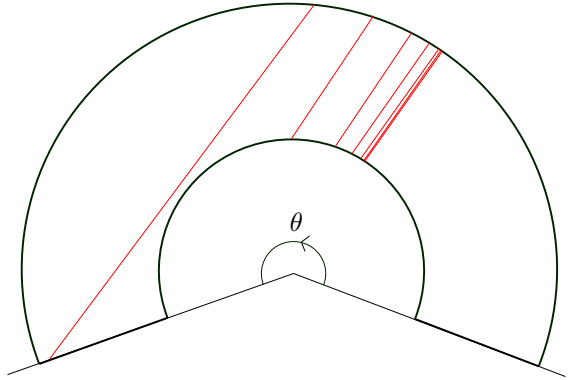
	\end{center}
	\caption{In black, a fundamental domain for the action of $z \mapsto 2 z$ on a cone of angle $\theta > \pi$. Any leaf entering the cylinder is trapped within it forever regardless of the direction of the leaf, see the red one drawn on the figure for instance. This property prevents a polygonation to 'connect' both sides of the cylinder.}
			\label{figgroscylindre}
\end{figure}

 A \textbf{straight line polygonation} of a dilation structure is a topological polygonation of $\Sigma$ with set of vertices equal to the set the singularities of the dilation structure and whose edges are straight lines connecting singularities (also called \textbf{saddle connections}).

\begin{definition}
 We say that a dilation structure is \textbf{polygonable} if it admits a straight line polygonation. If the topological surface under consideration has boundary components, we impose moreover that these components lies in the set of straight line edges of the polygonation.
\end{definition}

Note that the above definition implies that all boundary components must contain at least one singularity since we required that the vertices of the polygonation are singularities. This condition may seen artificial at first, but it prevents (for instance) one to glue arbitrary long cylinders on the boundaries which would make the moduli space artificially big. \\

Note also that a polygonation of a surface may be refined to get a straight line triangulation. We will at times say that a surface is \textbf{triangulable} if it has a straight line polygonation whose faces are all triangles. A dilation surface is polygonable if and only if it is triangulable. \\

As already mentioned, not all dilation structure carries such a polygonalisation. The following theorem gives a simple characterisation of dilation surfaces which are polygonable.

\begin{theoreme}{\cite{veech2,DFG}}
  A dilation surface admits a straight line triangulation if and only if it
  does not contain an embedded open affine cylinder of angle
  $\pi$.
\end{theoreme}

This theorem was proven by Veech in a set of unpublished notes \cite{veech2}. \\

As already mentioned, we shall restrict ourselves to polygonable dilation structures. The advantage of being polygonable is that one can recover the dilation structure by gluing together planar polygons. This data is referred to as \textbf{a polygonal model} of the dilation structure. Such models will be of crucial use in order to study the geometry of the moduli space. 

\subsection{Moduli spaces and its natural associated dynamics.} \label{subsecdefmoduli} We have now all the material required to introduce moduli spaces. The following discussion is a gathering of well known facts, see \cite{FM}. We keep it  short since the next section will provide us with an explicit description of the moduli space coming from the polygonal models. \\

Since we consider surfaces with boundary, in order to define a sensible moduli space, we shall fix the combinatoric of the singularity with respect to the boundary components. For what follows when considering a topological surface $\Sigma$ we suppose fixed the number of singularity on each of the boundary components. We refer to this data as an onto mapping $\varphi$ from the singularities to the boundary components. For example, a genus 1 surface with 2 boundary components, 2 singularities on one component and 2 on the other one would be considered different from the same topological surface but with 3 singularities on the first boundary component and 1 on the other	. \\

We denote by $\mathcal{TD}(g,n,\partial, \varphi)$ the set of all polygonable dilation structures of a topological surface of genus $g$ with $n$ singularities, $\partial$ boundary components and fixed combinatoric $\varphi$ up to dilation automorphisms isotopic to the identity. Given a topological surface $\Sigma$ it is well known that the modular group $\mathrm{Mod}(\Sigma)$, \textit{i.e.} the group of all diffeomorphisms fixing the marked points and the boundary up to those isotopic to the identity, acts properly and discontinuously on the Teichmüller space (see for example \cite[Section 12.3]{FM}). Since a dilation structure gives rise to a holomorphic structure, the quotient in the following definition is an actual orbifold. 

\begin{definition}
	 We call the \textbf{dilation moduli space} of a surface $\Sigma$ of genus $g$ with $n$ singularities, $\partial$ boundary components and a combinatoric $\varphi$ the following quotient space:
		$$ \mathcal{D}(g,n,\partial; \varphi) := \quotient{\mathcal{TD}(g,n,\partial, \varphi)}{\mathrm{Mod}(\Sigma)} \ .$$
\end{definition}

This article aims to investigate the dynamical systems that dilation moduli spaces naturally carry through one of the first non-trivial example. We shall denote by $\D := \D(1,1,1, *)$ the moduli space of a \textbf{one holed torus with one singularity}. Note that, by our definition of polygonable the singularity must lie in the unique boundary component and that the combinatoric $*$ is trivial. Note also that it forces the euclidean angle around the unique singularity to be $\pi$ by an Euler characteristic argument. \\

In Section \ref{secroom} we give a simple description of $\D$ using polygonal models.

\subsection{The $\mathrm{SL}_2(\mathbb{R})$-action and the Teichmüller flow.}

There is a natural $\mathrm{SL}_2(\mathbb{R})$-action defined on $\mathcal{D}(g,n,\partial)$. Formally the action can be defined as follows. Given a matrix $A \in \mathrm{SL}_2(\mathbb{R})$, we define the image of the dilation structure $(U, \varphi_U)_{U \in \mathcal{U}}$ as the structure given by the atlas $(U, A \cdot \varphi_{U})_{U \in \mathcal{U}}$, which is still a dilation structure. The normal subgroup $\mathbb{R}^*_+ \cdot \mathrm{Id} < \mathrm{GL}^+_2(\mathbb{R})$ lies in the kernel of this action and therefore we have a well-defined action of $\mathrm{GL}^+_2(\RR) / \mathbb{R}^*_+ \sim \mathrm{SL}_2(\RR)$. \\

One can also describe this action more visually using polygonal model that we introduced in the previous subsection: if a dilation structure is given by gluing parallel sides of a set $P$ of polygons together and if $A$ is an element of $\mathrm{SL}_2(\RR)$, the image of the corresponding dilation structure under $A$ is simply given by the unique corresponding gluing of the sets of polygons $A \cdot P$, see Figure \ref{figdefsl}. In other words, the $\mathrm{SL}_2(\RR)$-action on a moduli space $\D(g,n,\partial)$ is the trace of the natural $\mathrm{SL}_2(\RR)$-action on the space of polygons. \\

\begin{figure}[!h]
	\begin{center}
		\def\svgwidth{0.6 \columnwidth}
			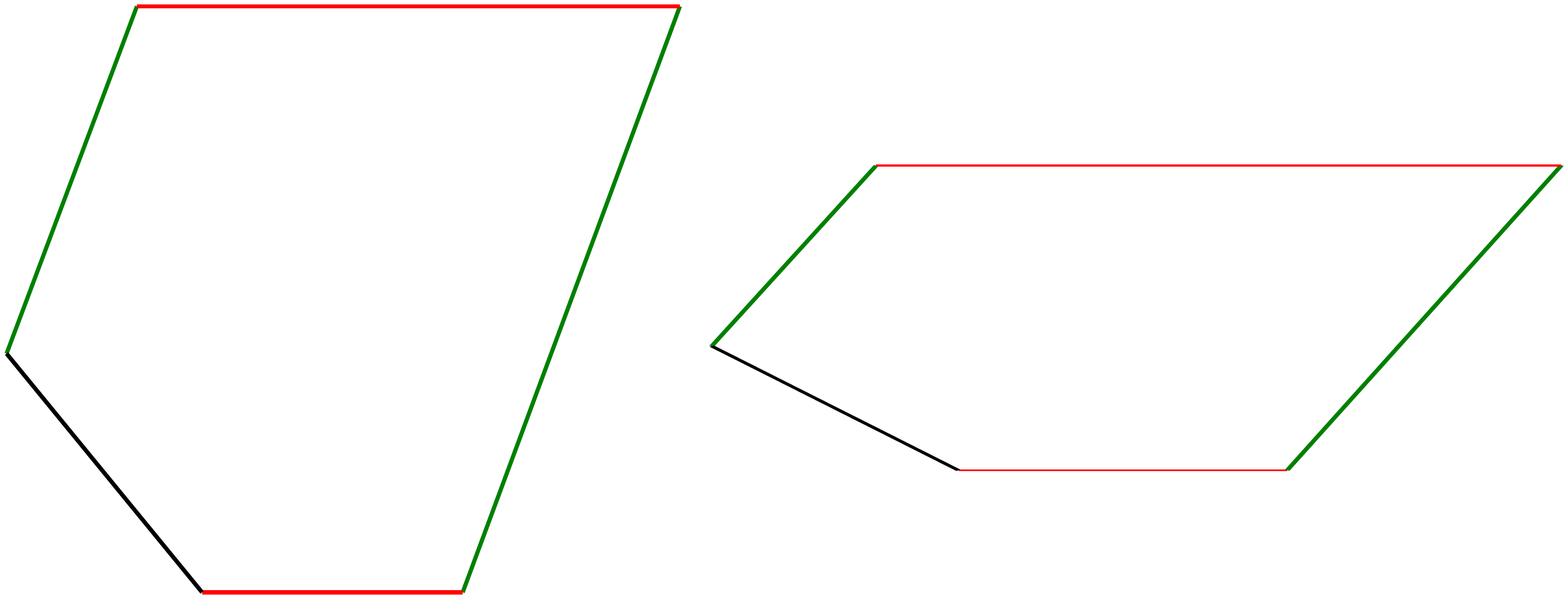
			\label{figdefsl}
	\end{center}
	\caption{The pentagon on the right is the image under a  diagonal matrix $A$ of $\SL_2(\RR)$ of the pentagonal on the left, up to a some dilation. The corresponding dilation structures obtained from these pentagonal models are also image from one another. by elements of $\SL_2(\RR)$. }
	\label{}
\end{figure}

One can also restrict the previous action to the diagonal group, parametrised as
	$$ \left\{ \begin{pmatrix}
		e^t & 0 \\ 0 & e^{-t} 
	\end{pmatrix} \right\}_{t \in \RR} \ ,  $$
in order to get an action of $\mathbb{R}$ on the moduli space, \textit{i.e.} a flow. This flow is called the \textbf{Teichmüller flow}. This flow is an extension to dilation surfaces of the standard on acting on the moduli space of translation surfaces. It was extensively studied since Masur's seminal work relating the recurrence of the orbit of a given translation surface to the unique ergodicity of the horizontal foliation of the same surface \cite{Masur}.

\section{The rooms}
\label{secroom}

\subsection{The pentagonal model}  We denote by $\Sigma_{1,1}$ a topological torus with one boundary component and a marked point lying on the boundary component. In the case of a dilation structure in $\D(1,1,1)$, meaning a dilation structure on the topological surface $\Sigma_{1,1}$, one can always have a one face polygonal model which embeds in the plane as in Figure \ref{figroom}.

\begin{lemma}
	\label{lempentagonal}
	A dilation structure of $\D(1,1,1)$ can always be obtained in gluing two pairs of parallel sides of a pentagon in the plane. In particular, every dilation structure may be represented by a convex polygon in the plane.
\end{lemma}

\begin{figure}[!h]
	\begin{center}
		\def\svgwidth{0.6 \columnwidth}
			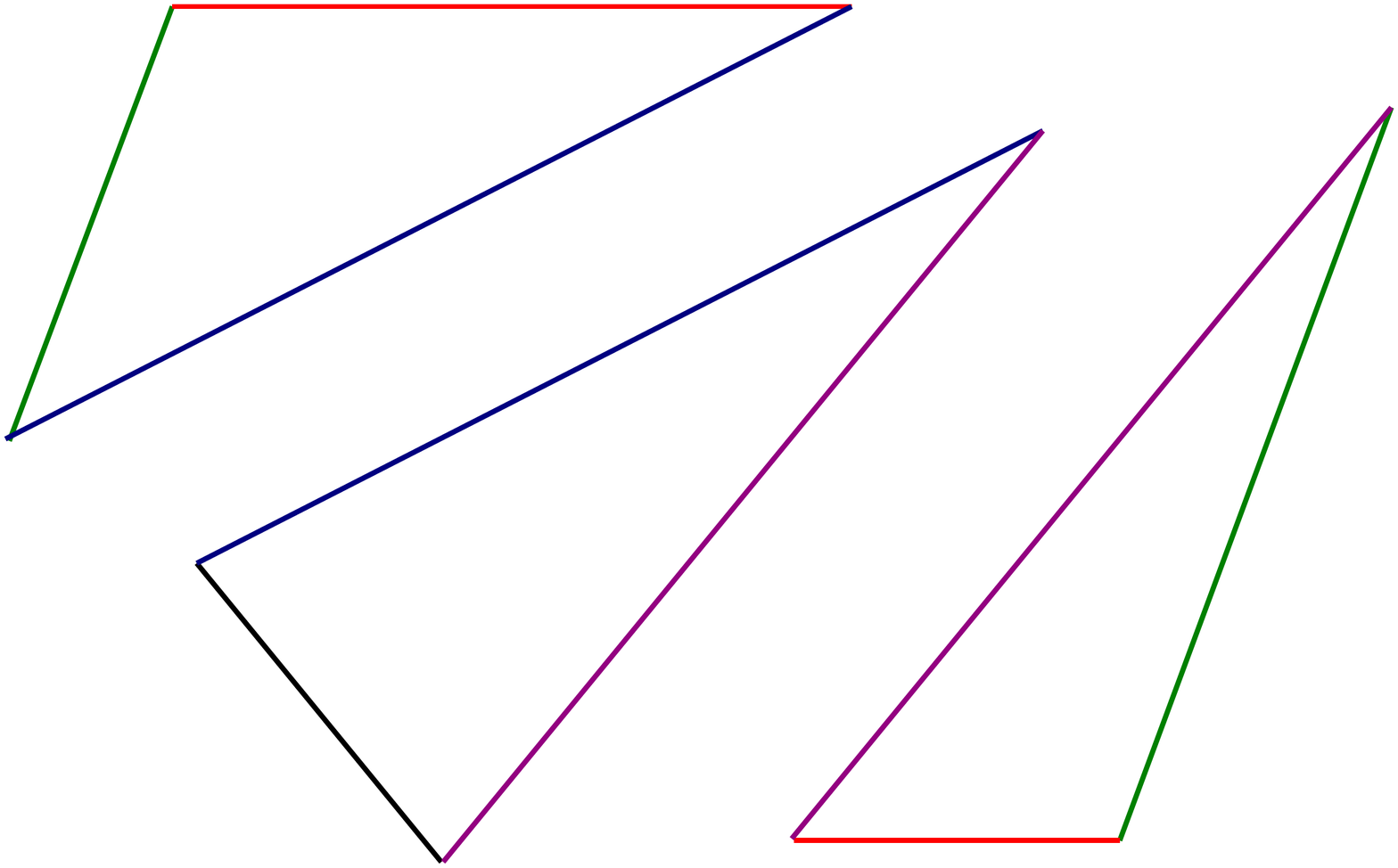
	\end{center}
	\caption{The three triangles are glued according to the colors of their side. The centred one has a black side corresponding to the boundary of $\Sigma_{1,1}$. The top and the right triangles have a gluing corresponding to the one of a torus without boundary if one identifies - topologically - the blue side and the violet one. This corresponds to collapsing the centre triangle.} 
\label{figroom3}
\end{figure}

\textbf{Proof.} Let $\mathfrak{D} \in \D$ a dilation structure, which is polygonable by definition $\mathfrak{D}$. Up to refining a given polygonation, one can suppose that it is a triangulation. We denote by $f$ the number of faces and and by $e$ the number of edges of such a triangulation. Using the Euler characteristic of the surface one can check that we must have on one hand
	$$ - 1 = f - e + 1 $$   
and, since every edges belongs to exactly 2 faces except the boundary, one the other hand 
	$$ 3 f = 2 e - 1$$   
which gives that $e = 5$ and then $f =3$. One it then left to triangles that, glued together, give rise to our dilation structure $\mathfrak{D}$. Since we required that the boundary component must be an edge of the triangulation, there is one and only one in our case, triangle having one of sides corresponding to the straight line boundary of our torus. Collapsing topologically the boundary side and the face of this triangle should give the gluing of two triangles which must give a triangulation of the torus without boundary. Therefore, inserting the collapsed triangle back between the two triangle left must lead to the gluing represented in Figure \ref{figroom3}. This gluing actually gives rise to a pentagonal as in Figure \ref{figroom}, concluding.  \hfill $\blacksquare$ \\

The following definition is the key of what will define our parametrisation of $\D$.
\begin{definition}
	We call the \textbf{space of rooms}, that we denote by $\mathcal{R}$, the space of all pentagonals appearing in Lemma \ref{lempentagonal} up to dilation by a positive number, see figure \ref{figroom}.
\end{definition}

\begin{figure}[!h]
	\begin{center}
		\def\svgwidth{0.6 \columnwidth}
			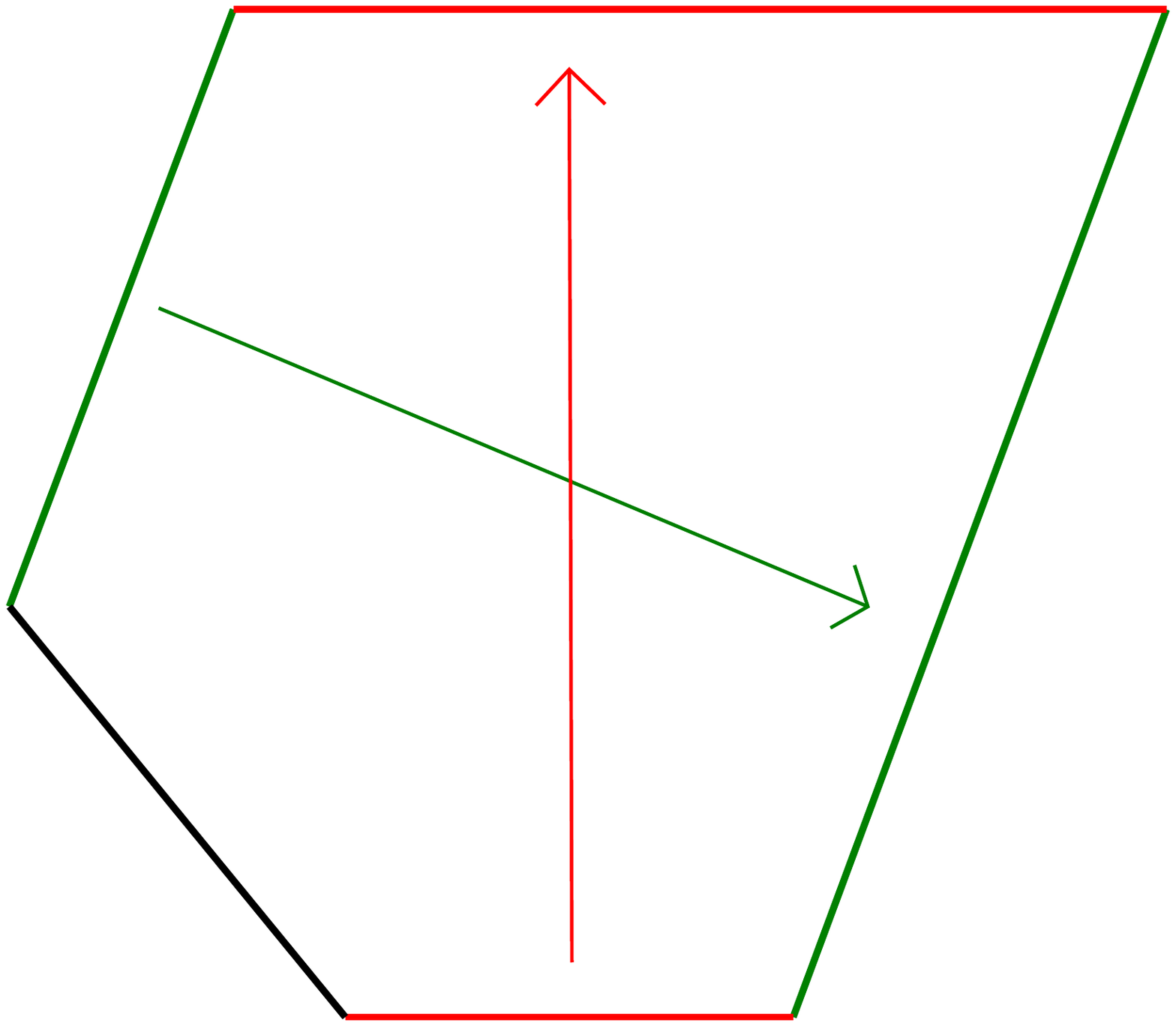
	\end{center}
	\caption{This polygonable model has a unique 2-cell, a pentagon, a unique vertex corresponding to the vertices of the pentagon and three edges. Topologically, the resulting surface is a torus with one boundary component, the black edge on the Figure. The green and red sides are glued using the two unique dilation maps (that we represented by the two arrows) sending a side to a parallel one.}
			\label{figroom}
\end{figure}

Lemma \ref{lempentagonal} implies that the map $\mathcal{R} \to \D$ which associates to a room its corresponding dilation structure is onto and continuous, in other words the space of rooms parametrises $\D$.  In the next subsection, we make this parametrisation more explicit by using a simple parametrisation of the space of rooms. 

\subsection{An explicit parametrisation of the space of rooms.}  Now that we have the pentagonal model introduced in the last subsection, we use it to get a nice parametrisation of our moduli space. The following picture shows how to associate to a room a basis, up to multiplicative real constant and a pair of 'dilation coefficients'.

\begin{figure}[!h]
	\begin{center}
		\def\svgwidth{0.6 \columnwidth}
			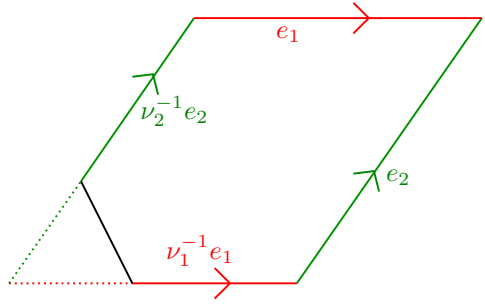
	\end{center}
	\caption{Given a oriented basis $(e_1, e_2)$ of $\mathbb{R}^2$ and two dilation parameters not both greater than $1$, one can build a room in considering the rectangle of sides $(e_1, e_2)$ here in red and green and remove a triangle so that the ratio between parallel sides are respectively $\nu_1^{-1}$ and $\nu_2^{-1}$.} 
\label{figroom4}
\end{figure}

The above Figure shows that one can recover a room, and therefore a dilation structure of $\D$, from an oriented basis taken up to dilations and a pair of dilation parameters. We denote by $\mathcal{B}$ the set of all oriented basis of $\mathbb{R}^2$ up to dilations. For the resulting pentagonal to be well described, one must require that not both $\nu_1 < 1$ and $\nu_2 < 1$ holds simultaneously. \\ 

We shall actually consider the pair $(\mu_1, \mu_2)= (\ln(\nu_1), \ln(\nu_2))$ instead of $(\nu_1, \nu_2)$ since it will appear that it is more convenient to work with. We call them the \textbf{log-dilation parameters}, and we denote by $\mathcal{Q}$ the subset of $\RR^2$ of all pairs $(\mu_1, \mu_2)$ such that $\mu_1 < 0$ and $\mu_2 < 0$ do not simultaneously hold. We also denote by $p$ the map which to a basis and pair of log-dilation parameters in $\mathcal{Q}$ gives back the naturally occurring room,
	$$ p := \mathcal{B} \times \mathcal{Q} \to \mathcal{R} \ ,$$ 
 described above and in figure \ref{figroom4}. \\

This map is onto and natural with respect to the $\SL_2(\RR)$-linear action on both $\mathcal{B}$ and the space of rooms; for any basis $\mathcal{B}$, and any log-dilation parameters $(\mu_1, \mu_2) \in \mathcal{Q}$ and any $A \in \SL_2(\RR)$ we have 
\begin{equation}
	\label{eqnaturelroom}
	 p (A \cdot \mathfrak{B}, (\mu_1, \mu_2) ) = A \cdot  p (\mathfrak{B}, (\mu_1, \mu_2) )  \ .
\end{equation}

Since we assumed that our dilation structure were polygonable Lemma \ref{lempentagonal} guarantees that any polygonable structure can be recovered from a room and therefore the parametrisation $p$ of the space of room gives rise to a well defined parametrisation $\tilde{p}$ to the moduli space $\D$:
	$$ \tilde{p} : \mathcal{B} \times \mathcal{Q} \to \D \ . $$ 

We call this map the "local dilation torus parametrisation". Moreover, the naturality of \eqref{eqnaturelroom} implies that this parametrisation is equivariant with respect to both the $\SL_2(\RR)$-action on rooms and on our moduli space $\D$:
given a matrix $A \in \SL_2(\RR)$ and parameters $ (\mathfrak{B} (\mu_1 , \mu_2)) \in \mathcal{B} \times \mathcal{Q}$ we have 
\begin{equation}
	\label{eqnatureldil}
	 \tilde{p} (A \cdot \mathfrak{B}, (\mu_1, \mu_2)) = A \cdot \tilde{p} (\mathfrak{B}, (\mu_1, \mu_2)) \ .
\end{equation}

This parametrisation takes into account in a simple way both the $\SL_2(\mathbb{R})$-action and the holonomy group of the dilation structure.

\begin{definition}
	The subgroup of $\RR_+$ generated by the dilation parameters is called the \textbf{linear holonomy group} associated to the dilation structure $\mathfrak{D}$.
\end{definition}

Since this group is the image of the underlined holonomy map it only depends on the underlined dilation structure $\mathfrak{D}$. \\ 

The local dilation torus parametrisation is not one-to-one since we did not take the mapping class group action into account. It remains to understand what pairs of pentagonal models give rise to the same dilation structure on a one-holed torus. The mapping class group of the torus with one boundary - homomorphic to $\SL_2(\ZZ)$ - writes down explicitly at the level of our space of parameters. \\

\begin{figure}[!h]
	\begin{center}
		\def\svgwidth{0.5 \columnwidth}
			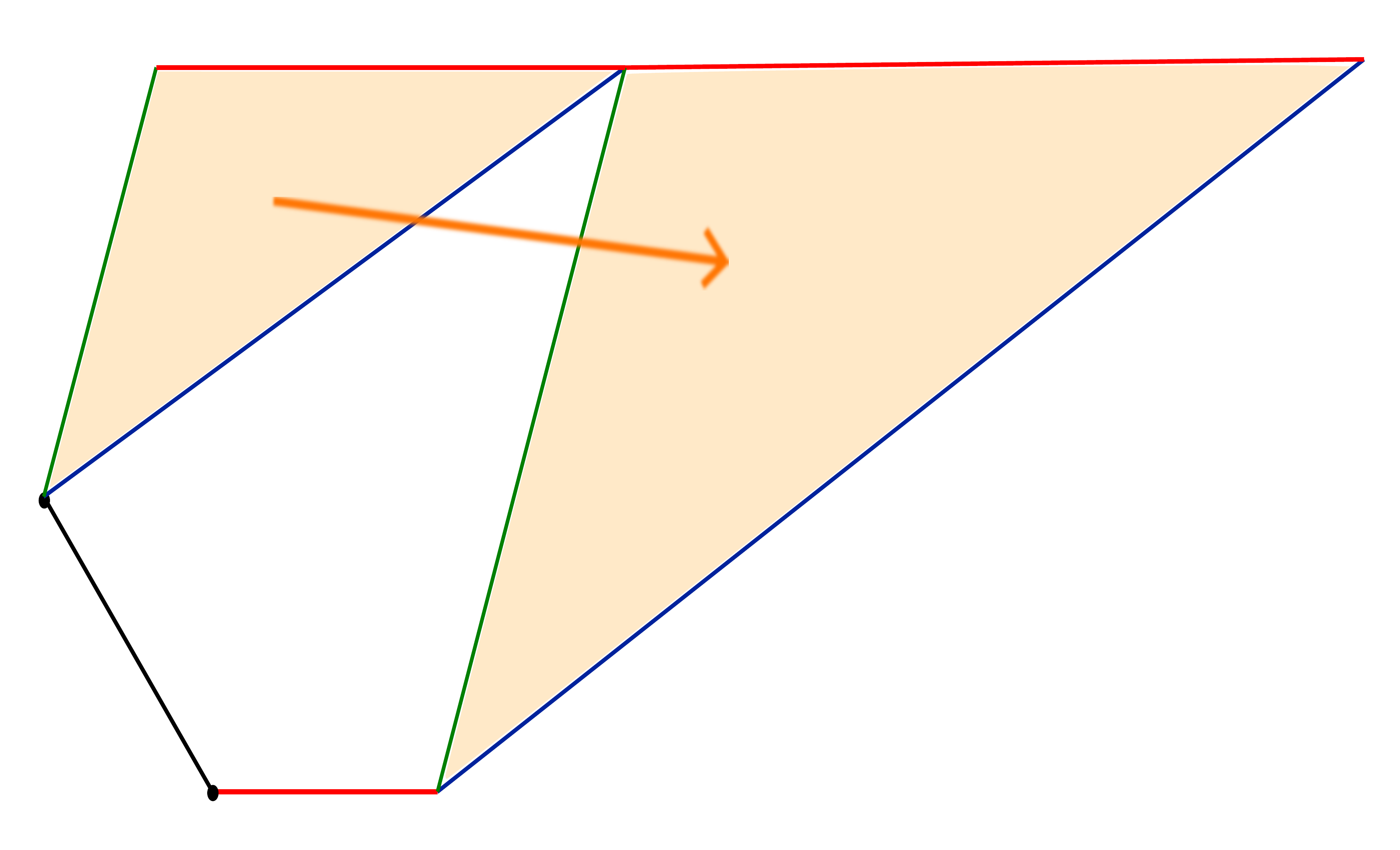
	\end{center}
	\caption{Cutting off the orange upper triangle and gluing it to the left does not change the underlying dilation structure but does change the associated pentagonal model.} 
\label{figmappingclass}
\end{figure}

The following formulas are obtained by cut and paste operations as illustrated in Figure \ref{figmappingclass}; $\tau_1$ and $\tau_2$ correspond to Dehn twists along the two curves generating the homology of $\Sigma_{1,1}$. Figure \ref{figroommappingclass2} shows how to get the first following formula, corresponding to the cut and paste movement described in Figure \ref{figmappingclass}.

\begin{equation}
	\label{eqformulemappingclass}
	\left\{  
		\begin{split}
	 \tau_1 ( e_1, e_2, \mu_1, \mu_2) = ((e_1 + \nu_1 e_2, \nu_1 e_2), \mu_1, \mu_1 + \mu_2) \\
	 	 \tau_2( e_1, e_2, \mu_1, \mu_2) = (\nu_2 e_1, e_2 + \nu_2 e_1, \mu_1 + \mu_2, \mu_2) \\
	 	 (\tau_1)^{-1} ( e_1, e_2, \mu_1, \mu_2) = (e_1 - e_2, \nu_1^{-1} e_2, \mu_1,  \mu_2- \mu_1)  \\
	 	 (\tau_2)^{-1} ( e_1, e_2, \mu_1, \mu_2) = (\nu_2^{-1}e_1, e_2 - e_1, \mu_1 - \mu_2, \mu_2)
		\end{split}
	\right.  
\end{equation}

\begin{figure}[!h]
	\begin{center}
		\def\svgwidth{ 0.5 \columnwidth}
			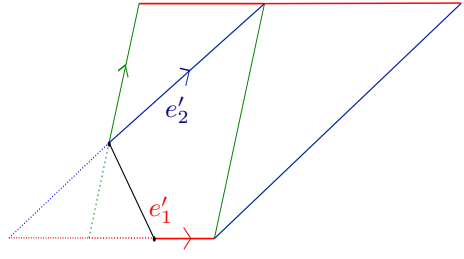
	\end{center}
	\caption{Since we dilated the upper left triangle by a factor $\nu_2$ in pasting it to the right, the vector $e'_1$ is $ \nu_2 e_1$. The vector $e'_2$ is given by adding the vector $e_2$ to the vector $e'_1 = \nu_2 e_1$: $e'_2 = e_2 + \mu_2 e_1$. The new dilation parameters are easy to compute by looking at similar triangles.} 
\label{figroommappingclass2}
\end{figure}

As already emphasised, the action on the log-dilation parameters corresponds to the action on the image of the holonomy representation of the dilation structure obtained by pre-composing with a given element of the mapping class group, seen as acting on the first homology groups. Forgetting the basis component of our parameters this action writes down, as expected from the general theory, simply as the linear action of $\SL_2(\ZZ)$ on $\mathcal{Q} \subset \RR^2$.  \\

Note however that given a room one can not always perform the 4 movements described above: we must make sure that both the resulting dilation parameters remains in $\mathcal{Q}$ in order to be sure that the resulting data defines a dilation structure. In order to take this into account at the level of our space of parameters, we introduce the following equivalence relation: two points $x,y \in \mathcal{B} \times \mathcal{Q}$ are said to be equivalent (we note $x \sim y$) if there is a sequence $\tau^{\epsilon_1}_{i_1}, ..., \tau^{\epsilon_n}_{i_n}$, with $\epsilon_i \in \{-1,1\}$ and $i_j \in \{1,2\}$ such that  
\begin{itemize}
	\item  $ \tau^{\epsilon_n}_{i_n} \circ ... \circ \tau^{\epsilon_1}_{i_1}(x) = y$;
	\item for all $ 1 \le j \le n-1$, $ \tau^{\epsilon_j}_{i_j} \circ ... \circ  \tau^{\epsilon_1}_{i_1} (x)$ belongs in $\mathcal{B} \times \mathcal{Q}$.
\end{itemize}

We will say that two rooms are $\sim$\textbf{equivalent} if there are equivalent under the relation $\sim$. \\

Because of technicalities that we will encounter along the proof of the proposition stated below, we weaken this equivalence relation by forcing the log-dilation parameters to remain in $\RR_+^2$ instead of $\mathcal{Q}$ and exchanging the above second condition defining the equivalence relation $\sim$ by requiring that the sequence $\tau^{\epsilon_j}_{k_j}, ..., \circ  \tau^{\epsilon_1}_{k_1} (x)$ has log-dilations parameters remaining in $\RR_+^2$ all along the process. Since every surface can be represented by a convex room from lemma \ref{lempentagonal}, one will still have a parametrisation. We shall call the sequence described above an \textbf{admissible paths of rooms}. The following Proposition sums up elements of the above discussion:

\begin{proposition}
	\label{propparametrisation}
	The map
		$$ \mathcal{P} \ : \ \quotient{\mathcal{B} \times \RR_+^2}{\sim}  \to \D \ , $$ 
is continuous, onto and equivariant with respect to the $\SL_2(\RR)$-actions already defined on $\D$ and $\mathcal{B}$. 
\end{proposition}

The map just defined will be called the \textbf{dilation torus parametrisation}. 

\begin{remark}
	The equivalence relation just introduced might not seem very natural. One might wonder if $\D$ is actually homomorphic to $$ \quotient{\mathcal{B} \times \RR^2 \setminus \{ 0 \}}{\SL_2(\ZZ)} \ ,$$
where the action of $\SL_2(\ZZ)$ is defined on 
$\mathcal{B} \times \RR^2 \setminus \{0\}$ by extending its action using the formulae given by \eqref{eqformulemappingclass}. This would be a satisfactory  algebraic description of $\D$ as it would make of $\mathcal{B} \times \RR^2 \setminus \{ 0 \}$ and $\SL_2(\ZZ)$ its fundamental group.
\end{remark}

\subsection{The $\SL_2(\RR)$-action}

The following proposition corresponds to the first half of our main theorem. We say that a dilation structure $\mathfrak{D}$ is \textbf{irrational} if the pair of dilation parameters given by any room representing $\mathfrak{D}$ are not rationally proportional.

\begin{proposition}
	\label{propsl2dense}
The $\SL_2(\RR)$-orbit of any irrational dilation structure is dense in $\D$. 
\end{proposition}

In particular, almost every dilation structure of $\D$ (both in the measurable and in the topological sense) has a dense $\SL_2(\RR)$-orbit. \\

\textbf{Proof.} We shall use the dilation torus parametrisation of $\D$ given by Proposition \ref{propparametrisation}. Since the dilation torus parametrisation is equivariant with respect to the $\SL_2(\RR)$-actions and since $\mathcal{P}$ is continuous and onto one would get the desired conclusion if one would be able to prove that the $\SL_2(\RR)$-orbit of any irrational room $\mathfrak{R}_0 = (\mathfrak{B}_0, \mu_0)$ is dense in the space of rooms $\quotient{ \mathcal{B} \times \RR_+^2}{\sim}$. Which could be rephrased as to show that the image of the set of rooms $\sim$equivalent to $\mathfrak{R}_0$ in the quotient space 
	$$ \mfaktor{\SL_2(\RR)}{\mathcal{B} \times \RR_+^2}$$ 
is dense. \\

Since the $\SL_2(\RR)$-action defined on the set of rooms does not affect the log-dilations parameters and is transitive on $\mathcal{B}$, one is left to show that the image in $\RR_+^2$ of the set of rooms $\sim$equivalent to $\mathfrak{R}_0$ is dense. We will slightly abuse the notation in keeping denoting by $\sim$ the relation in $\RR_+^2$ defined as $x \sim y$ if and only if there is two rooms $\sim$equivalent for which $x$ and $y$ are the log-dilation parameters. \\

The key remark is that an admissible path between two $\sim$equivalent rooms projects on $\RR_+^2$ as a path from which two successive points are related if and only if they are image of one another by an element of 
\begin{equation}
\label{eqrefgeneratingset}
\hspace{0.5 cm}  \begin{pmatrix}
1 & 1 \\
0 & 1 
\end{pmatrix}^{\pm 1}  \hspace{0.5 cm}  \begin{pmatrix}
1 & 0 \\
1 & 1 
\end{pmatrix}^{\pm 1} \ .  
\end{equation} 

Therefore, the question under investigation turns out to be very similar to the more classic fact that the $\SL_2(\ZZ)$-orbit of any point in the place with  rationally independent coordinates $\mu_0$ is dense in $\RR_+^2$, expect that we are only allow to use matrices which can be decomposed in the Cayley graph of $\SL_2(\ZZ)$ (with respect to the generating set \eqref{eqrefgeneratingset}) in a way that all the matrices obtained along this path send $\mu_0$ to $\RR^2_+$.  \\

We split the proof into three  steps: first we show that this set accumulates to $0_{\RR^2}$ from which we will deduce that it also accumulates to $(\RR_+,0)$ and we will conclude by showing that it implies that it is dense in $\RR_+^2$. \\

Note first that Gauss' algorithm gives rise to an admissible path: given any $\mu_0 = (x_0, y_0)$ such that $x_0 / y_0 \notin \QQ$ one can construct an infinite sequence $\mu_n \in \RR_+^2$ inductively as follows. If $\mu_i = (x_i, y_i)$ with $x_i > y_i$ (the other case being treated symmetrically) we define $k_i$ as the larger integer such that $ x_i - k_i y_i > 0 $. We then defined $\mu_{i+1} = (x_i - k_i y_i, y_i)$. Note that the algorithm is well defined and that $\mu_n \tends{n \to \infty} 0_{\RR_2}$ since we suppose $x_0 / y_0 \notin \QQ$. Note also that one can rewrite $\mu_n$ as the projection on the log-coordinates of the element 
	$$ \tau^{- k_j}_{i_j} \circ ... \circ  \tau^{-k_1}_{i_1} (\mathfrak{R}_0) \ , $$
where the index $i_j \in \{1,2\}$ is defined depending on whether $x_{i} > y_i$ or $x_i < y_i$. To sum up, we just showed that the set $ \{ \mu \in \RR_+^2 \ , \ \mu \sim \mu_0 \} $ accumulates to $0_{\RR^2}$. \\

Given $\mu = (\epsilon_1, \epsilon_2)$ ( $\epsilon_1$ and $\epsilon_2$  are thought of very small, as the output of Gauss' algorithm after a large number of steps) one can construct an admissible path from $\mu$ to $ (\epsilon_1, \epsilon_2 + n \epsilon_1)$ in applying $\tau_1$ $n$-consecutive times. As a consequence, the set $ \{ \mu \in \RR_+^2 \ , \ \mu \sim \mu_0 \} $  accumulates on the line $(\RR_+,0)$ since it accumulates to  $0_{\RR^2}$. \\

To conclude, let us remark that any element of $\SL_2(\NN)$ can be decomposed in positive powers of the two following matrices (using the fact that the coefficients of same columns and lines are relatively prime numbers and in using Gauss' algorithm again)
	$$ \begin{pmatrix}
	1& 1 \\ 0 & 1 
	\end{pmatrix} \hspace{1 cm} \begin{pmatrix}
	1& 0 \\ 1 & 1 
	\end{pmatrix} \ ,$$
Any sequence of powers of these two matrices always gives rise to an admissible path for starting point $\mu_0$  (of positive coordinates) since all product of such matrices have positive entries. As a consequence the set $ \{ \mu \in \RR_+^2 \ , \ \mu \sim \mu_0 \} $ accumulates on $\SL_2(\NN) \cdot (\RR_+,0)$, which is dense in $\RR_+^2$, concluding.
 \hfill $\blacksquare$

\section{The directional foliations on the torus with boundary}
\label{directionalfoliations}

\noindent In the present section $\Sigma$ is a fixed dilation torus with boundary. 

\subsection{From directional foliations to maps of the interval}

Recall that $\Sigma$ carries a family $\mathcal{F}_{\theta}$ of transversally affine foliations indexed by angle $\theta \in S^1$. In this paragraph we discuss how, in the case of tori with boundary, the study of these foliations can be reduced to that of certain piecewise continuous maps of the interval.

\vspace{3mm} We say that a direction $\theta \in S^1$ is \textbf{pointing inwards}  if no trajectory in direction $\theta$ meets the interior of the door of the room. Concretely, if the door is in the vertical direction and if the room is lying on the right of the door, the set of direction pointing inwards is $]-\frac{\pi}{2}, \frac{\pi}{2}[$. The directions of the door we call the \textbf{door direction}.  
\noindent Recall that, by Lemma \ref{lempentagonal}, a (triangulable) dilation torus with boundary can be recovered from the gluing of a pentagonal model. Choose such a model for $\Sigma$. Up to an extra sequence of gluing and pasting triangles, this model can be assumed to be convex, as in the Figure \ref{figroom7} below;  

\noindent Consider the diagonals of this pentagon (there are $5$ such diagonals). These project onto closed curves in the $\Sigma$. For any $\theta$ pointing inwards we have the following properties:

\begin{enumerate}
\item at least one of these curves is transverse to $\theta$;

\item the first-return map of $\mathcal{F}_{\theta}$ to this curve is well-defined;

\item this first-return map is piecewise continuous with exactly one discontinuity point;

\item this first-return map is orientation preserving and affine restricted to its interval of continuity.
\end{enumerate}

\begin{figure}[!h]
	\begin{center}
			\def\svgwidth{0.8 \columnwidth}
			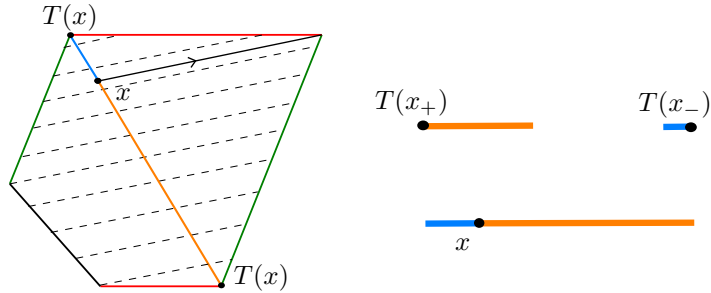
	\end{center}
	\caption{Here, $T$ denotes the first return map on the blue-orange segment associated to the foliation represented by dashed black lines. The point $x$ splits the segment in two pieces (corresponding to the colors) on which the map $T$ acts by dilation. On the right, the associated affine interval exchange representation: on the bottom we represented the partition of the segment with respect to which the map is piecewise affine and on the top the image of such a partition under the map $T$. Note that $x$ is a continuity point of $T$ (as a map from the circle to itself) in this example but not as a map from the segment to itself.}
\label{figroom7}
\end{figure}

We introduce the following definition:

\begin{definition} Let $\rho_A$ and $\rho_B$ two positive number not simultaneously larger than $1$. A $(\rho_A, \rho_B)$-map of the interval $I = [0,1]$ is a map $T$ satisfying the following conditions:

\begin{itemize}
\item $T$ has exactly one discontinuity point $x_T \in ]0,1[$;

\item $T$ is injective

\item $T$ is affine with slope $\rho_A$ restricted to $A = [0,x_T[$ and $\lim_{x \rightarrow x_T^-}{T(x)} = 1$;

\item $T$ is affine with slope $\rho_B$ restricted to $B = ]x_T,1]$  and $\lim_{x \rightarrow x_T^+}{T(x)} = 0$.
\end{itemize}

\noindent Such an $(\rho_A, \rho_B)$-map is uniquely determined by the point $x_T$ and can therefore be parametrised canonically by an affine parameter that we can normalise to take values in $[0,1]$. 
\end{definition}

\vspace{3mm}
\paragraph{\bf Reducing the problem to $(\rho_A, \rho_B)$-maps} The discussion above should have convinced the reader that to each directional foliation we can associate a map $T$ that is close to being a $(\rho_A, \rho_B)$-map, but which isn't always one. The two last conditions of the definitions are not always satisfied. However, one can move from a piecewise affine and continuous map with one discontinuity to a $(\rho_A, \rho_B)$-map by considering the smallest connected interval containing the image of $T$. It is stable, $T$ restricted to such an interval (an suitably rescaled) is a $(\rho_A,\rho_B)$-map and it sees all the dynamics as all the points excluded are mapped by $T$ to it after a single iteration.

\begin{figure}[!h]
	\begin{center}
			\def\svgwidth{0.8 \columnwidth}
			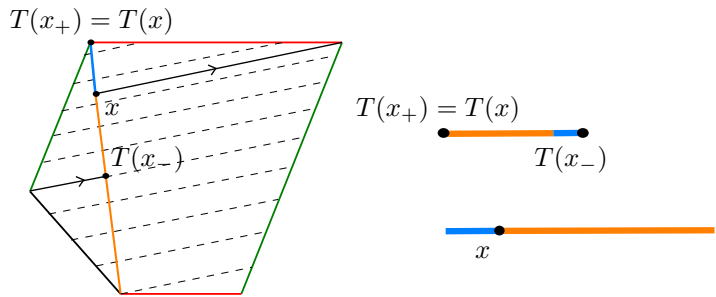
	\end{center}
	\caption{This picture is analogous to Figure \ref{figroom7} but with a different cross-section. The associated interval exchange is not a $(\rho_A,\rho_B)$ map but becomes one in restriction to the image, as shown in Figure \ref{figechangeinter}.}
\label{figroom8}
\end{figure}

\begin{figure}[!h]
	\begin{center}
			\def\svgwidth{0.5 \columnwidth}
			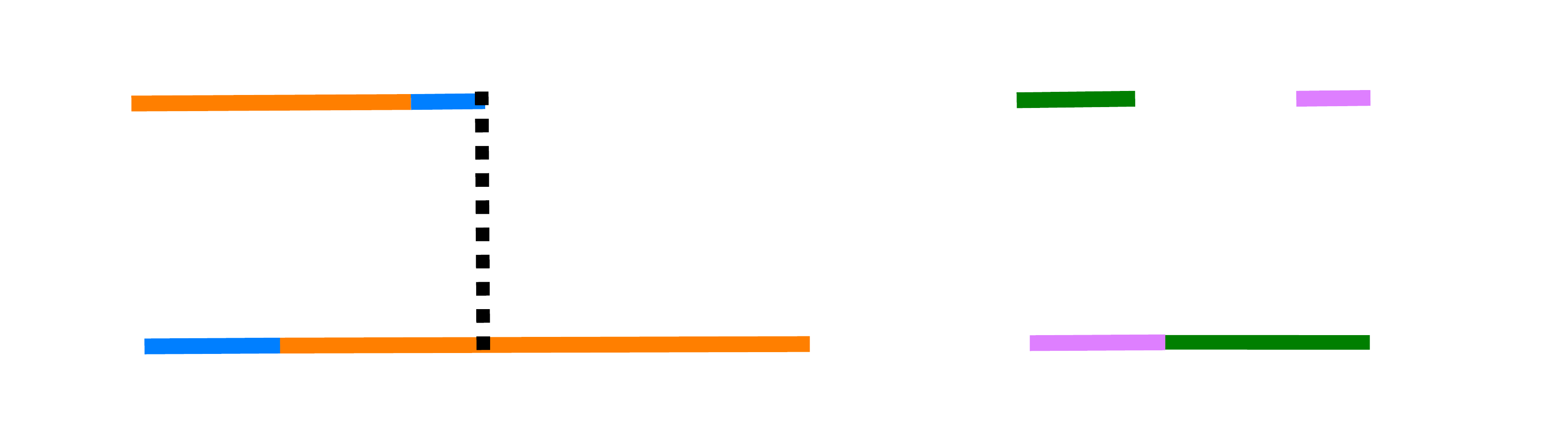
	\end{center}
	\caption{We chop off the interval with respect to the black dotted line (according to the image of $T$). The resulting map, for which the affine interval exchange representation is drawn on the right, is a $(\rho_A,\rho_B)$-map.}
\label{figechangeinter}
\end{figure}

In a sense that we are going to make precise, the study of the directional foliations $(\mathcal{F}_{\theta})$ reduces exactly to that of $(\rho_A,\rho_B)$-maps. We introduce 

$$\mathcal{E}= \mathcal{E}_{\rho_A,\rho_B} = \{(\rho_A, \rho_B)\text{-maps} \} $$ which naturally identifies to an interval, since for example the length of $A$ completely determined the map, which we will consider normalised to be $[0,1]$. If the first-return map of $\mathcal{F}_{\theta_0}$ on a diagonal $D$ yields a $(\rho_A,\rho_B)$-map $T(\theta)$ for certain positive numbers $\rho_A$ and $\rho_B$ it is also the case for a $U$ neighbourhood of $\theta_0 \in S^1$. The induced map

$$ U \longrightarrow \mathcal{E}_{\rho_A,\rho_B}  $$ is smooth. Therefore the analysis of the dynamical behaviour of foliations in the family $(\mathcal{F}_{\theta})_{\theta \in S^1}$ can be reduced to that of countably many families $\mathcal{E}_{\rho_A,\rho_B} $.

\subsection{Rauzy induction for contracting maps}
In this paragraph we introduce a renormalisation scheme to describe the dynamics of $(\rho_A, \rho_B)$-maps.

\vspace{3mm} \noindent We introduce a renormalisation scheme that we call \textit{Rauzy induction} although it is not strictly speaking the standard Rauzy induction used for the study of interval exchange transformation, see Figure \ref{figechangeinter2}. We start with $T$ a $(\rho_A, \rho_B)$-map.

\begin{enumerate}

\item if $x_T$ belongs to $T(A)$, then \textbf{the algorithm returns the first return map o}n $A$ and we say that $A$ is the \textit{winner};
\item if $x_T$ belongs to $T(B)$, then \textbf{the algorithm returns the first return map on} $B$ and we say that $B$ is the \textit{winner};
\item if $x_T$ does not belong to $T([0,1])$ then the \textbf{algorithm stops}.
\end{enumerate}

\begin{figure}[!h]
	\begin{center}
			\def\svgwidth{0.8 \columnwidth}
			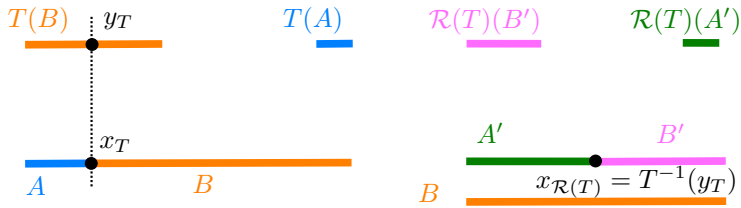
	\end{center}
	\caption{Here $x_T$ belongs to $T(B)$, so that $B$ wins and we consider the first return map on the interval $B$. The affine interval transformation corresponding to the resulting map $\mathcal{R}(T)$ is drawn on the right part of the above picture. In this case the algorithm stops after only one iteration since $x_{\mathcal{R}(T)}$ does not belong neither to $\mathcal{R}(A')$ nor to $\mathcal{R}(B')$.}
\label{figechangeinter2}
\end{figure}

Suppose we are in one of the first two cases. We denote by $\mathcal{R}(T)$ the induced first-return maps. We have the following properties

\begin{itemize}

\item if $A$ is the winner then $\mathcal{R}(T)$ is $(\rho_A, \rho_A\rho_B)$-map;

\item if $A$ is the winner then $\mathcal{R}(T)$ is $(\rho_A \rho_B, 
\rho_B)$-map;

\item if $x_T$ does not belong to $T([0,1])$ then $T$ has an attracting orbit of period $2$.

\end{itemize}

We explain in the following paragraphs how this algorithm allows for a complete description of the dynamics of elements of $\mathcal{E}_{\rho_A,\rho_B} $.

\vspace{3mm}

\paragraph{\bf The case $\rho_A, \rho_B < 1$} 

Recall that $\mathcal{E}_{\rho_A,\rho_B} $ identifies naturally with $[0,1]$. We can look at the subset of this interval corresponding to the three possible outcomes of the above induction. The set $\mathcal{E}_{\rho_A,\rho_B} $ is split into three connected intervals corresponding to the following situations 

\begin{enumerate}

\item the interval on the left $\mathcal{E}_{\rho_A,\rho_B}(L)$  corresponding to $B$ being the winner;

\item the interval on the right $\mathcal{E}_{\rho_A,\rho_B}(R)$  corresponding to $A$ being the winner;

\item the interval in the middle $\mathcal{E}_{\rho_A,\rho_B}(H)$ corresponding to the case when the algorithm stops.

\end{enumerate}

\noindent The latter corresponds to the case where the associated $(\rho_A, \rho_B)$-maps have an attracting periodic orbit of order $2$. In that case the basin of attraction of this attracting orbit is the whole interval. \\

\noindent In the two other cases, the algorithm defines two one-to-one maps 
$$ \mathcal{E}_{\rho_A,\rho_B}(L)  \longrightarrow  \mathcal{E}_{\rho_A\rho_B,\rho_B}   $$ and 

$$ \mathcal{E}_{\rho_A,\rho_B}(R)  \longrightarrow  \mathcal{E}_{\rho_A, \rho_A\rho_B}   $$ and the analysis of the dynamics of elements of $ \mathcal{E}_{\rho_A,\rho_B}(L) $ and $ \mathcal{E}_{\rho_A,\rho_B}(R)$   reduces to that of their images via the above identifications.

 \vspace{2mm} This short discussion provides in this case the inductive step for the construction of a Cantor set. Indeed both $ \mathcal{E}_{\rho_A,\rho_B}(L)$ and $ \mathcal{E}_{\rho_A,\rho_B}(R)$ can be subdivided further into three subintervals. In both case the subinterval in the middle corresponds to the existence of an attracting periodic orbit whereas the left and right interval corresponds to element for which the algorithm can be continued and these subinterval identify with entire parameters spaces $\mathcal{E}_{\rho_A',\rho_B'}$ for new values $(\rho_A', \rho_B')$. This way we get 

\begin{itemize}

\item each finite word $W$ in the alphabet $\{R,L\}$ corresponds to a subinterval of  $\mathcal{E}_{\rho_A,\rho_B}(W)$ which identifies via the algorithm to an entire parameter space $\mathcal{E}_{\rho_A(W),\rho_B(W)}$ for a certain pair $(\rho_A(W),\rho_B(W))$;

\item the intersection of all the possible $\mathcal{E}_{\rho_A,\rho_B}(W)$  forms a Cantor set $\mathcal{R}_{\rho_A,\rho_B}$;

\item elements in the complement of this Cantor set all have a unique attracting periodic orbit whose basin of attraction is the entire interval;

\item all but countably many elements of $\mathcal{R}_{\rho_A,\rho_B}$ are infinitely renormalisable (meaning that they can be applied the induction algorithm infinitely many times).

\end{itemize}

This discussion is a summary of a more detailed analysis carried out in the article \cite{BFG}[Section 4 5]. Therein is proved the following theorem.

\begin{theoreme}[\cite{BFG}[Section 4 5]
\label{description}
Assume  $\rho_A, \rho_B < 1$.
\begin{enumerate}

\item The subset of $\mathcal{E}_{\rho_A,\rho_B}$ of infinitely renormalisable maps is contained in a certain Cantor set of Lebesgue measure $0$.
\item The complement of this Cantor set in made of $(\rho_A,\rho_B)$-maps which have a unique attracting periodic orbit. 

\end{enumerate}
\end{theoreme}

For the proof of this Theorem we refer to the article \cite{BFG}. Therein the theorem is only proved for $\rho_A = \rho_B = \frac{1}{2}$ but the proof generalises verbatim to the case $\rho_A, \rho_B < 1$.

\vspace{3mm}
\paragraph{\bf The general case}  One can try to implement the same algorithmic procedure to analyse $\mathcal{E}_{\rho_A,\rho_B}$  either $\rho_A$ or $\rho_B$ is larger than $1$. Then main difference in this case is that out of the three possible cases for the first step of the algorithm, one need not necessarily occur. Indeed, if $\rho_A > 1$ and $\rho_A \rho_B > 1$, $A$ cannot be the winner. For if it were, the algorithm would yield a $(\rho_A, \rho_A \rho_B)$-map which cannot be, as both $\rho_A$ and $\rho_A \rho_B$ would be larger than $1$. We leave it to the reader to check that the following Proposition holds true

\begin{proposition}
Assume $\rho_A >1$. We have the following alternative.

\begin{itemize}

\item Either $\rho_A\rho_B <1$ in which case  $\mathcal{E}_{\rho_A,\rho_B}(L)$,  $\mathcal{E}_{\rho_A,\rho_B}(R) $ and  $\mathcal{E}_{\rho_A,\rho_B}(H)$ are non-empty.

\item Or $\rho_A\rho_B >1$ in which case $ \mathcal{E}_{\rho_A,\rho_B} = \mathcal{E}_{\rho_A,\rho_B}(L)$.
\end{itemize}

\end{proposition}

With the Proposition at hand, we can derive a picture for the iteration of the induction algorithm very similar to that of the case $\rho_A, \rho_B < 1$. 

\begin{itemize}

\item Assume $\rho_A >1$ and $\rho_A \rho_B < 1$. In that case  $\mathcal{E}_{\rho_A,\rho_B}$ is split into three subintervals. One of which (the middle one) corresponds to periodic orbits and for whose elements the algorithm stops. Elements of $\mathcal{E}_{\rho_A,\rho_B}(R) $ identify via the algorithm to $\mathcal{E}_{\rho_A\rho_B,\rho_B} $ and we are reduced to the case $\rho_A$ and $\rho_B < 1$. Finally $\mathcal{E}_{\rho_A,\rho_B}(L) $ is still of the form $\rho_A < 1$.

\item If $\rho_A \rho_B > 1$ we can apply the algorithm a finite number of steps until  $\mathcal{E}_{\rho_A,\rho_B}$ identifies to  $\mathcal{E}_{\rho_A',\rho_B'} $ for $\rho_A' \rho_B' < 1$ (each step decreases $\rho_A$ by multiplying it by $\rho_B$). We are thus reduced to the step above.

\end{itemize}

Note that up to the acceleration of the second case, the picture is exactly the same as in the case $(\rho_A, \rho_B)$. After $n$ steps of induction, there are $2^n$ subintervals left (by that we mean that we have chucked out intervals corresponding to periodic orbits). These correspond to $n$-times renormalisable elements of  $\mathcal{E}_{\rho_A,\rho_B}$ and each of them identifies with a set $\mathcal{E}_{\rho_A(W),\rho_B(W)}$ where $W$ is the word in $L$ and $R$ to which it corresponds. Furthermore, the only interval in this collection for which $\rho_A(W)$ is larger than one is the leftmost one, and within the others can be replicated the picture of the case $\rho_A, \rho_B < 1$. From this description we deduce that the conclusions of Theorem \ref{description} still hold true in this case:

\begin{proposition}
\label{description2}
Assume  $\rho_A >1$ and $\rho_B < 1$.
\begin{enumerate}

\item The subset of $\mathcal{E}_{\rho_A,\rho_B}$ of infinitely renormalisable maps is contained in a certain Cantor set of Lebesgue measure $0$.
\item The complement of this Cantor set in made of $(\rho_A,\rho_B)$-maps which have a unique attracting periodic orbit. 

\end{enumerate}
\end{proposition}

\subsection{Estimates on the size of the hole}

In order to prove convergence or divergence of orbits of the geodesic flow, we will some quantitative estimates on the size of intervals of parameters corresponding to periodic orbits of a given combinatorics.

\begin{lemma}
\label{estimates}
We have 

\begin{itemize}

\item $ | \mathcal{E}_{\rho_A,\rho_B}(R) | = \frac{\rho_A}{1 + \rho_A}$;

\item $ | \mathcal{E}_{\rho_A,\rho_B}(L) | = \frac{\rho_B}{1 + \rho_B}$;

\item $ \mathcal{E}_{\rho_A,\rho_B}(H) $ is non-empty if and only if $\rho_A \rho_b < 1$ in which case $ | \mathcal{E}_{\rho_A,\rho_B}(H) | = \frac{1- \rho_A\rho_B}{(1 + \rho_A)(1 + \rho_B)}$.

\end{itemize}

\end{lemma}

The important information that this Lemma tells us, beyond the formulae is the following 

\begin{itemize}

\item if $\rho_A$ and $\rho_B$ are simultaneously very small, then the "hole" is very big (its size in $[0,1]$ is about $1- \rho_A\rho_B$);

\item if $\rho_A$ and $\rho_B$ are both very close to $1$, then the hole is very small (again, its size is proportional to $1- \rho_A\rho_B$), its position is close to $\frac{1}{2}$ and is therefore negligible before both $\mathcal{E}_{\rho_A,\rho_B}(L)$ and $\mathcal{E}_{\rho_A,\rho_B}(R)$.

\end{itemize}

\subsection{Door directions and Herman's family}

The analysis of door directions is reduced by means of first return maps to the special family of piecewise affine circle diffeomorphisms with two discontinuities of the derivative and such that one of these two discontinuity is mapped to the other. Such a map is completely determined by the data $\rho_A, \rho_B)$ and corresponds to an extremal point of $\mathcal{E}_{\rho_A,\rho_B}$. 

\noindent This family has appeared in \cite[p79, Section 7.3]{Herman1}. In \cite{LiousseMarzougui}, the authors prove that the map which associate to $(\rho_A,\rho_B)$ the rotation number of the associated circle homeomorphism is analytic. This implies that for almost every surface in $\mathcal{D}$, the flow in the door direction is minimal (because the associated first return map has irrational rotation number). 
\noindent Moreover, it is explained in \cite{LiousseMarzouhui} that in this family every element is conjugate to the rotation of same rotation number via a piecewise analytic map. This strongly suggests that orbit of the Teichmüller flow of an element of $\mathcal{D}$ in the door direction should accumulate a moduli space of translation tori (which lies in the boundary of $\mathcal{D}$ and which corresponds to collapsing the boundary component).

\section{The Teichmüller flow}
\label{secTeich}

The goal of this Section is to establish the following Theorem

\begin{theoreme}

For any $T$ in $\mathcal{D}$ we have that $g_t(T)$ diverges.

\end{theoreme}

\subsection{Two divergence criteria}

In this paragraph we give two criteria to establish divergence of a sequence in $\mathcal{D}$.

\begin{proposition}
\label{divergence1}
Let $(T_n)$ be a subsequence of $\mathcal{D}$ such that 

$$ \Theta(T_n) \longrightarrow 0 \ \text{or} \ \pi$$ as $n$ goes to infinity. Then $T_n$ diverges.

\end{proposition}

\begin{proof}
This is a simple consequence of the fact that the function $\Theta$ is continuous on $\mathcal{D}$ an takes its values in $]0, \pi[$.

\end{proof}

\begin{proposition}
\label{divergence2}
Let $(T_n)$ be a subsequence of $\mathcal{D}$ such that each $T_n$ contains a cylinder $C_n$ such that there exists $\epsilon>0$ such that the following holds:

\begin{enumerate}

\item for all $n \in \mathbb{N}$, $\Theta(C_n) > \epsilon$;

\item $\rho(C_n)  \longrightarrow +\infty$ as $n$ goes to infinity.

\end{enumerate}

Then $(T_n)$ diverges.

\end{proposition}

\begin{proof}

Recall that a given dilation surface has only finitely many cylinders of angle greater than a fixed positive constant $\epsilon$. If a sequence satisfying the hypothesis of the Proposition had a convergent subsequence, (up to extracting a subsequence) the cylinder $C_n$ should converge to a cylinder in the limit surface of angle larger than $\epsilon$ (since there are only finitely such cylinders). Which would imply convergence of $\rho(C_n)$ to the multiplier of the cylinder in the limit surface, hence contradicting the fact that 
$\rho(C_n)  \longrightarrow +\infty$.

\end{proof}

%
%
%
%
%
%
%
%
%
%
%

\subsection{Action of $g_t$ on $S^1$}

The Teichmüller flow $g_t$ acts of the set of directional foliations of a given dilation tori in a natural way. Indeed, there is a natural action of $\mathrm{SL}_2(\mathbb{R})$ on $S^1$ (which is the projective action on the set of semi-lines of $\mathbb{R}^2$ which identifies with $S^1$), and if $\mathcal{F}_{\theta}(T)$ denotes the foliation in direction $\theta$ on $T$ we have that for any $A \in \mathrm{SL}_2(\mathbb{R})$, 

$$ \mathcal{F}_{\theta}(T) \simeq  \mathcal{F}_{g \cdot \theta}(A \cdot T).$$

\noindent In particular we have the following property: if $T$ has a cylinder covering the interval of directions $[\theta_1, \theta_2]$, then $A\cdot T$ has a cylinder of same multiplier covering the set of directions   $[A \cdot \theta_1, A \cdot \theta_2]$.

\vspace{3mm} We now restrict our attention to the action of $g_t = \begin{pmatrix}
e^{-\frac{t}{2}} & 0 \\
0 & e^{\frac{t}{2}}
\end{pmatrix})_{t \in \mathbb{R}}$. The projective action of this one-parameter family preserves the interval of directions $[-\frac{\pi}{2}, \frac{\pi}{2}]$. For $t >0$, $0$ is a repelling fixed point of $g_t$ and $\frac{\pi}{2}$. We give a distortion Lemma which we will use later on.

\begin{lemma}
\label{distortion}
There exists a constant $C>1$ such that for any $t >0$ the following holds. Denote by $I$ the preimage of $[-\frac{\pi}{4}, \frac{\pi}{4}]$ by $g_t$ thought of as a Moebius diffeomorphism of $[-\frac{\pi}{2}, \frac{\pi}{2}]$. Then for any $x$ and $y$ in $I$ we have 

$$ C^{-1} \leq \frac{D(g_t)(x)}{D(g_t)(y)} \leq C.$$ In other words, for any $t$ the distortion of $g_t$ on $g_t^{-1}([-\frac{\pi}{4}, \frac{\pi}{4}])$ is uniformly bounded. 
\end{lemma} 

\begin{proof}
Up to a fixed change of coordinates, the family $g_t$ is smoothly conjugate away from $\frac{\pi}{2}$ and $-\frac{\pi}{2}$ to the family $x \mapsto \lambda x$ which has uniformly bounded distortion.

\end{proof}

\subsection{Divergence under the geodesic flow}

We now turn to proving the main theorem of this Section.

\begin{theoreme}

Let $T \in \mathcal{D}$. The orbit of $T$ under the action of the  Teichmuller flow $g_t$ eventually leaves all compact sets of $\mathcal{D}$.

\end{theoreme}

We distinguish three cases for the proof:

\begin{enumerate}
\item directions for which trajectories are ultimately trapped within a dilation cylinder and accumulate on a periodic orbit (cylinder directions);

\item directions which accumulate on a transversally Cantor set (Cantor directions);

\item directions parallel to the door (Door directions). 

\end{enumerate}

We have proven in Section \ref{directionalfoliations} that these three cases exhaust the dynamical possibilities. \\

\paragraph*{\bf Cylinder directions}

We consider $T$ such that all orbits of the directional foliation in the horizontal direction accumulate onto an attracting periodic orbit, which is to say that the associated first return $(rho_A,rho_B)$-map has a stopping-in-finite-time renormalisation scheme.

This periodic orbit is contained in a cylinder $C_0$ which comprises the horizontal direction. The image of a cylinder containing the horizontal direction in its interior under the action of the geodesic flow is a sequence of cylinders whose angles converge to $\pi$. Using Proposition \ref{divergence1} we see that $g_t(T)$ diverges. \\

\begin{figure}[!h]
	\begin{center}
			\def\svgwidth{0.5 \columnwidth}
			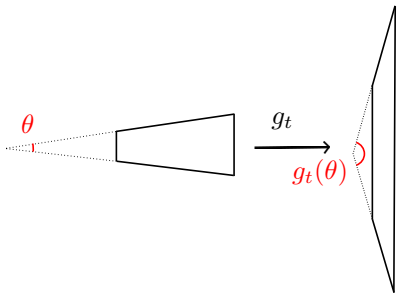
	\end{center}
	\caption{On the left, a fundamental domain for the actior of $z \mapsto 2 z$ of the angular sector of angle $\theta$. On the left, its image under the linear matrix $g_t$, which is again a fundamental domain for $z \mapsto 2 z$ but on the angular sector $g_t(\theta)$, where we abused the notation in also denoting by $g_t$ the projective action.}
\label{figcylindre3}
\end{figure}

\vspace{3mm}

\paragraph*{\bf Cantor directions}

Cantor directions are slightly more complicated to analyse. Consider $T$ such that its horizontal foliation as an invariant Cantor set. We can choose a small neighbourhood of $0 \in S^1$ for which the corresponding foliations can be identified by means of a first return map to an interval of a family $\mathcal{E}_{\rho_A,\rho_B}$ as we have done in Section \ref{directionalfoliations}, with both $\rho_A$ and $\rho_B$ strictly smaller than $1$.

\vspace{2mm} Now, consider the action of the geodesic flow $g_t$. It induces an action on the set of directions which has the following properties: $0\in S^1$ (the horizontal direction) is a repelling fixed point of $g_t$ for $t > 0$ and the derivative at $0$ is $\exp(t)$. Thus, if one wants to understand the behaviour of the sequence $g_t \cdot T$, one has to understand the dynamics of the foliations of angle very close to $0$. This is achieved using the induction introduced in Section \ref{directionalfoliations}. 
\noindent We now work with a small interval $I$ of parameters in $\mathcal{E}_{\rho_A,\rho_B}$  which identifies to a neighbourhood of the horizontal direction in $S^1$. Let $e_0$ the point in $\mathcal{E}_{\rho_A,\rho_B}$ corresponding to $T$. The Rauzy induction provides us with a sequence of nested intervals $(I_n)_{n \in \mathbb{N}}$ with the following properties

\begin{enumerate}

\item $I_{n+1}$ is defined inductively by choosing the left or right interval defined by Rauzy induction to which $e_0$ is the parameter corresponding to the horizontal foliation on $T$ belongs.

\item For all $n$, $I_n \subset  \mathcal{E}_{\rho_A,\rho_B}$;

\item $I_{n+1} \subset I_n$;

\item $I_{n}$ is an interval corresponding to a step in the Rauzy induction;

\item $\bigcap_n{I_n} = e_0$.
\end{enumerate}

Lemma \ref{estimates} shows that because the multipliers associated to the families of intervals $\mathcal{E}_{\rho_A,\rho_B}$ exhausted by the induction tend to $0$, the proportional size of the hole contained within $I_{n}$ tends to $1$ as $n$ goes to infinity.

\vspace{2mm} Consider the interval $g_t^{-1}([-\frac{\pi}{4}, \frac{\pi}{4}]) = [-\alpha(t),\alpha(t)]$ with $\alpha(t) \rightarrow 0$ when $t \rightarrow + \infty$.  On this interval, $g_t$ acts with bounded distortion. For any $t$ there exists $n(t)$ that $I_{n(t)+1} \subset [-\alpha(t),\alpha(t)] \subset I_{n(t)}$. \\

Thus for $t$ large enough, $ [-\alpha(t),\alpha(t)] $ contains a cylinder of angle $\theta(t)$ such that the ratio 

$$ \frac{\theta(t)}{|[-\alpha(t),\alpha(t)]|} $$ is uniformly bounded from below (this cylinder is either the one whose corresponding angular sector is contained within $I_{n(t)} \setminus I_{n(t)+1}$ or within $I_{n(t)+1} \setminus I_{n(t)+2}$.

Lemma \ref{distortion} ensures that the ratio 

$$\frac{\tilde{\theta}(t)}{\frac{\pi}{2}} $$ where $\tilde{\theta}(t)$ is the angle of the image of thus cylinder in the surface $g_t(T)$ is also uniformly bounded. We have therefore established that any time $t>0$ there is a cylinder in $T$ whose angular sector is contained in a neighbourhood of zero which is zoomed out by $g_t$ to a cylinder of angle uniformly bounded below. But when $t$ tends to infinity, the multiplier of such a cylinder tends to infinity (this is a consequence of the discussion about the Rauzy induction on $\mathcal{E}_{\rho_A,\rho_B}$ for $\rho_A$ and $\rho_B < 1$). Lemma \ref{divergence2} therefore implies that the sequence $g_t(T)$ diverges.

\vspace{3mm}
\paragraph*{\bf Door direction}

Recall that the "door direction" is the direction of the boundary component of   $T$. We assume in this paragraph that it is the horizontal direction. We identify a neighbourhood on the right of $0 \in S^1$ to a neighbourhood of the leftmost point in $\mathcal{E}_{\rho_A,\rho_B} \simeq [0,1]$ for a $\rho_A>1$ and $\rho_B < 1$.  The discussion on Rauzy induction yields the following fact

\begin{proposition}

Let leftmost point in  $\mathcal{E}_{\rho_A,\rho_B}$ (corresponding to the door direction) is accumulated by a sequence of nested intervals $(J_n)_{n\in \mathbb{N}}$ such that each $J_n$ identifies with a $\mathcal{E}_{\rho_A^n,\rho_B^n}$ with $(\rho_A^n,\rho_B^n) \rightarrow (1,1)$ when $n \rightarrow +\infty$.

\end{proposition}

Applying Lemma \ref{estimates} to pairs $(\rho_A,\rho_B)$ close to $(1,1)$ we can deduce the following:

\begin{proposition}

For any $\epsilon>0$, there exists a neighbourhood $U_{\epsilon}$ of $0\in S^1$ such that for any cylinder whose angular sector is $[\theta_1,\theta_2] \subset U_{\epsilon}$ we have

$$  \frac{|[\theta_1,\theta_2]|}{[0, \theta2]|} \leq \epsilon.$$

\end{proposition}

Applying Lemma \ref{distortion}, we obtain that $\Theta(g_t(T)) \longrightarrow 0 $ when $t$ tends to $+ \infty$. By Proposition \ref{divergence1} we get that the sequence $g_t(T) $ diverges. This concludes the proof of Theorem

\section{Comments and open problems}

\label{comments}

We conclude this article with a few discussions on related problems and open problems. \\

\paragraph*{\bf The $\mathrm{SL}_2(\mathbb{R})$-action} Understanding the dynamical properties of the action of $\mathrm{SL}_2(\mathbb{R})$ is key to the understanding of fine geometric properties of dilation surfaces. For instance, it is expected that closed, connected and $\mathrm{SL}_2(\mathbb{R})$-invariant sets corresponds to dilation surfaces sharing special geometric properties.  \\

\noindent We know of two sources of remarkable $\mathrm{SL}_2(\mathbb{R})$-invariant sets:

\begin{enumerate}

\item triangulability and its variations, which gives rise to natural open $\mathrm{SL}_2(\mathbb{R})$-invariant sets;

\item linear holonomy with value in a discrete subgroup of $(\mathrm{R}_+, \times)$ which defines closed invariant sets.

\end{enumerate}

We also know of non-trivial "Veech surfaces" (see \cite{BFG}) and it seems reasonable to suspect the existence of other type of $\mathrm{SL}_2(\mathbb{R})$-orbit closures. A first vague open problem is 

\begin{problem}
Classify $\mathrm{SL}_2(\mathbb{R})$-orbit closures for the $\mathrm{SL}_2(\mathbb{R})$-action on the moduli space of dilation surfaces. 
\end{problem}

A more particular problem that we think to be of interest is that of the existence of a dense orbit. Of course, the $\mathrm{SL}_2(\mathbb{R})$-action preserves triangulability and the singularity type. Bearing this in mind, we formulate the following conjecture:

\begin{conjecture}
\label{conj1}
Let $\mathcal{T}$ be the open subset of a stratum of dilation surfaces of genus at least $2$, consisting of triangulable dilation surface. Then there exists a $\Sigma \in \mathcal{T}$ such that 

$$\overline{ \mathrm{SL}_2(\mathbb{R}) \cdot \Sigma} = \mathcal{T}.$$
\end{conjecture}

It might be that in some cases, this conjecture has to be slighted modified to take into account the existence of invariant open sets based on refinement of the trinagulability property as \cite{Tahar} indicates.
\vspace{3mm}

\paragraph*{\bf Degenerations} Another direction of research we think is interesting is that of degenerations of dilation surfaces. We know of three different ways for a sequence of dilation surfaces to degenerate. Consider a sequence of dilation surfaces $(\Sigma_n)_{n \in \mathbb{N}}$.

\begin{enumerate}

\item There is a "door" which is collapsed;

\item Each $\Sigma_n$ contains a cylinder $C_n$ of angle $\theta_n \geq \epsilon$ and multiplier $\rho_n \rightarrow \infty$;

\item Each $\Sigma_n$ contains a cylinder $C_n$ of modulus (see \cite{Ghazman2}) which tends to $\infty$.

\end{enumerate}

We pose the following problem:

\begin{problem}
Are there other ways to degenerate that the three listed above?
\end{problem}

\vspace{3mm}

\paragraph*{\bf Genus $2$ surfaces}

It is possible to glue two dilation tori with boundary as considered in this article to form a genus $2$ surface with one singularity of cone angle $6\pi$. The space $\mathcal{R}$ of all such surfaces actually forms a connected components on the stratum of dilation surfaces of genus with one singular point. Such surfaces appear in \cite{DFG,BowmanSanderson}.
\noindent $\mathcal{R}$ is actually isomorphic to the product of two copies of the moduli space of the one-holed tori. Two easy corollary of the results of  the present article are the following:

\begin{enumerate}
\item  all orbits of Teichmüller flow in $\mathcal{R}$ are divergent;

\item a complete description of the dynamics of the directional foliations on those surfaces.

\end{enumerate}

An interesting question an answer to which would further our understanding of dilation surfaces is that of describing the $\mathrm{SL}_2(\mathbb{R})$-action on $\mathcal{R}$. In this case, it can be reduced to studying the diagonal action of triangular matrices on a product of two moduli spaces of one-holed tori for which the boundary direction is horizontal. The existence of a dense $\mathrm{SL}_2(\mathbb{R})$-orbit on $\mathcal{R}$ is equivalent to proving the topological mixing of the $\mathrm{SL}_2(\mathbb{R})$-action on $\mathcal{D}$. It would be a good test of Conjecture \ref{conj1}.

\vspace{3mm}

\paragraph*{\bf Invariant measure} We end this series of comments by mentioning a structural question that is of importance to the authors. In \cite{Ghazman}, the second author draws an analogy between moduli spaces of dilation surfaces and infinite volume hyperbolic manifolds. This analogy is supported by the existence in some low-dimensional cases of a $\mathrm{SL}_2(\mathbb{R})$-invariant measure, of infinite volume, which is equivalent to the Lebesgue measure. The existence (or non-existence) of such a measure is still unknown. We believe the existence of such a measure to be of capital importance as it would make the aforementioned analogy robust enough to prove interesting theorems about the dynamics of the Teichmüller flow, and would give an ergodic-theoretic framework for the study of the $\mathrm{SL}_2(\mathbb{R})$-action. \\

\bibliographystyle{plain}
	\bibliography{biblio.bib}

\end{document}